\documentclass{amsart}

\usepackage{inputenc}
\usepackage[T1]{fontenc}
\usepackage{amsmath}
\usepackage{amsthm}
\usepackage{amsfonts}
\usepackage{amssymb}
\usepackage{enumerate}
\usepackage{mathrsfs}
\usepackage{graphicx}
\usepackage{hyperref}
\usepackage[labelfont=rm]{caption}
\usepackage{subcaption}
\usepackage{tikz}
\usepackage{braids}
\usepackage{defs}
\newif\iftikziii
\tikziiifalse

\iftikziii
\usetikzlibrary{matrix,calc,arrows.meta,knots,intersections,decorations.markings,cd}
\else
\usetikzlibrary{matrix,calc,intersections,decorations.markings}
\fi
\usepackage{listings}

\title{Some open questions on arithmetic Zariski pairs}
\author[E. Artal]{Enrique Artal Bartolo}
\address{Departamento de Matem\'aticas, IUMA\\ 
Universidad de Zaragoza\\ 
C.~Pedro Cerbuna 12\\ 
50009 Zaragoza, Spain} 
\email{artal@unizar.es} 

\author[J.I. Cogolludo]{Jos{\'e} Ignacio Cogolludo-Agust{\'i}n}
\address{Departamento de Matem\'aticas, IUMA\\ 
Universidad de Zaragoza\\ 
C.~Pedro Cerbuna 12\\ 
50009 Zaragoza, Spain} 
\email{jicogo@unizar.es}

\thanks{Partially supported by
MTM2013-45710-C2-1-P}  



\begin{document}

\begin{abstract}
In this paper, complement-equivalent arithmetic Zariski pairs will be exhibited answering in the negative a 
question by Eyral-Oka~\cite{Oka-Eyral-fundamental-groups} on these curves and their groups.
A complement-equivalent arithmetic Zariski pair is a pair of complex projective plane curves having Galois-conjugate equations in 
some number field whose complements are homeomorphic, but whose embeddings in $\PP^2$ are not.

Most of the known invariants used to detect Zariski pairs depend on the \'etale fundamental group. In the case
of Galois-conjugate curves, their \'etale fundamental groups coincide. 
Braid monodromy factorization 
appears to be sensitive to the difference between \'etale fundamental 
groups and homeomorphism class of embeddings.
\end{abstract}

\maketitle

\section*{Introduction}

In this work some open questions regarding Galois-conjugated curves and arithmetic Zariski pairs will be 
answered and some new questions will be posed. The techniques used here combine braid monodromy calculations,
group theory, representation theory, and the special real structure of Galois-conjugated curves. 

A Zariski pair~\cite{ea:jag} is a pair of plane algebraic curves 
$\mathcal{C}_1,\mathcal{C}_2\in\mathbb{P}^2\equiv\mathbb{C}\mathbb{P}^2$ whose embeddings in their regular 
neighborhoods are homeomorphic $(T(\mathcal{C}_1),\mathcal{C}_1)\cong (T(\mathcal{C}_2),\mathcal{C}_2)$
but their embeddings in $\mathbb{P}^2$ are not $(\mathbb{P}^2,\mathcal{C}_1)\not\cong (\mathbb{P}^2,\mathcal{C}_2)$. 
The first condition is given by a discrete set of invariants which we refer to as \emph{purely combinatorial} in 
the following sense. The \emph{combinatorics} of a curve $\mathcal{C}$ with irreducible components 
$\mathcal{C}^1,\dots,\mathcal{C}^r$ is defined by the following data:
\begin{enumerate}
\enet{($\mathcal{C}$\arabic{enumi})}
\item The degrees $d_1,\dots,d_r$ of  $\mathcal{C}^1,\dots,\mathcal{C}^r$.
\item The topological types $\mathcal{T}_1,\dots,\mathcal{T}_s$ of the singular points
$P_1,\dots,P_s\in\sing(\mathcal{C})$.
\item Each $\mathcal{T}_i$ is determined by the topological types $\mathcal{T}_i^1,\dots,\mathcal{T}_i^{n_i}$
of its local irreducible branches $\delta_i^1,\dots,\delta_i^{n_i}$ and by the local intersection numbers 
$(\delta^j_i,\delta^k_i)_{P_i}$ of each pair of irreducible branches.
The final data for the definition of the combinatorics is the assignment of its global irreducible component
for each local branch~$\delta_i^j$.
\end{enumerate}
The first Zariski pair was found by O.~Zariski in \cite{zr:29,zr:37} and it can be described as  
a pair of irreducible sextics with six ordinary cusps. This example has two main features. On the one hand, the 
embeddings of the curves in $\PP^2$ are not homeomorphic because their complements are not. On the other hand,
one of the curves of the pair satisfies a nice global algebraic property (which is not part of its combinatorics): 
its six singular points lie on a conic. The first fact can be proved directly by showing that the fundamental 
groups of their complements are not isomorphic. Also, using \cite{zr:31} it is possible to prove this by means of
a weaker, but more tractable, invariant which was later called the Alexander polynomial of the curve by 
Libgober~\cite{li:82} which is sensitive to global aspects such as the position of the singularities. The second
feature 
is the fact that one of the sextics is a curve of \emph{torus type} i.e. a curve whose equation is of the form 
$f_2^3+f_3^2=0$, where $f_j$ is a homogeneous polynomial of degree~$j$.

Since then, many examples of Zariski pairs (and tuples) have been found by many authors, including 
J.~Carmona, A.~Degtyarev, M.~Marco, M.~Oka, G.~Rybnikov, I.~Shimada, H.~Tokunaga, and the authors, 
(see~\cite{act:08} for precise references).

By the work of Degtyarev~\cite{deg:90} and Namba~\cite{nmb:86}, Zariski pairs can appear only in degree at least~$6$, 
and this is why the literature of Zariski pairs of sextics is quite extensive. Given a pair of curves, it is usually 
easy to check that they have the same combinatorics. What is usually harder is to prove is whether or not they are 
homeomorphic. Note that two curves which admit an equisingular deformation are topologically equivalent and this is 
why the first step to check whether a given combinatorics may admit Zariski pairs is to find the connected components 
of the space of realizations of the combinatorics. Namely, given a pair of curves with the same combinatorics, a 
necessary condition for them to be a Zariski pair is that they are not connected by an equisingular deformation,
in the language of Degtyarev, they are not \emph{rigidly isotopic}. 

Most of the effective topological invariants used in the literature to prove that a pair of curves is a Zariski pair 
can be reinterpreted in algebraic terms, in other words, they only depend on the algebraic (or \'etale) fudamental 
group, defined as the inverse limit of the system of subgroups of finite index of a fundamental group.
This is why, in some sense when it comes to Zariski pairs, the most difficult candidates to deal with are those of 
an arithmetic nature, i.e. curves $\mathcal{C}_1,\mathcal{C}_2$ whose equations have coefficients in some number 
field~$\mathbb{Q}(\xi)$ and they are Galois conjugate. Note that Galois-conjugate curves have the same \'etale 
fudamental group.

There are many examples of pairs of Galois-conjugate sextic curves which are not rigidly isotopic. The first example 
of an arithmetic Zariski pair was found in~\cite{acc:01b} in degree 12 and it was built up from a pair of 
Galois-conjugate sextics (see also~\cite{Abelson-conjugate,kulikov-kharlamov,se:64} for similar examples on compact surfaces).

In another direction, the equivalence class of embeddings, i.e. the homeomorphism class of pairs $(\PP^2,\cC)$,
can be refined by allowing only homeomorphisms that are holomorphic at neighborhoods of the singular points of the
curve (called \emph{regular} by Degtyarev~\cite{degt:08}). Also, one can allow only homeomorphisms that can be 
extended to the exceptional divisors on a resolution of singularities. Curves that have the same combinatorics and
belong in different classes are called \emph{almost}-Zariski pairs in the first case and \textsc{nc}-Zariski
pairs in the second case.

Interesting results concern these other Zariski pairs, for instance Degtyarev proved in~\cite{degt:08} that sextics 
with simple singular points and not rigidly isotopic are \emph{almost} Zariski pairs, and among them there are plenty 
of arithmetic pairs. 

Shimada developed in~\cite{shi:08,shi:09} an invariant denoted $N_C$ which is a topological invariant of the 
embedding, but not of the complements. He found the first examples of arithmetic Zariski pairs for sextics.
None of these examples is of torus type.

In \cite{Oka-Eyral-fundamental-groups}, Eyral and Oka study a pair of Galois-conjugated curves of torus type.
They were able to find presentations of the fundamental groups of their complements and was conjectured that these 
groups are not isomorphic, in particular this would produce an arithmetic Zariski pair. The invariant used by 
Shimada to find arithmetic Zariski pairs of sextics does not distinguish Eyral-Oka curves. Also, 
Degtyarev~\cite{MR2513586} proposed alternative methods to attack the problem, but it is still open as originally 
posed by Eyral-Oka.

This paper answers some questions on the Eyral-Oka example. The first part of the conjecture is solved in the 
negative by proving that the fundamental groups of both curves are in fact isomorphic. The question about them being 
an arithmetic Zariski pair remains open but, using the techniques in~\cite{act:08}, several arithmetic Zariski pairs 
can be exhibited by adding lines to the original curves. It is right hence to conjecture that they form an arithmetic 
Zariski pair themselves. Moreover, some of these Zariski pairs are \emph{complement-equivalent Zariski pairs}, 
(cf.~\cite{act:08}) that is, their complements are homeomorphic (actually analytically and algebraically isomorphic 
in this case) but no homeomorphism of the complements extends to the curves.

Also, a very relevant fact about these curves that makes computations of braid monodromies, and hence fundamental 
groups, very effective from a theoretical point of view is that they are not only real curves, but 
\emph{strongly real curves}, that is, their singular points are also real plus the real picture and the combinatorics
are enough to describe the embedding. Some of the special techniques used for strongly real curves were outlined 
in~\cite{acct:01}. In this work we will describe them in more detail.

The paper is organized as follows. In Section~\ref{sec-construction} the projective Eyral-Oka curves will be constructed. 
Their main properties are described and one of the main results of this paper is proved: after adding a line to the 
projective Eyral-Oka curve we obtain the affine Eyral-Oka curve and we show that their complements are homeomorphic
in \autoref{prop-iso}. Section~\ref{sec-okaeyral} is devoted to giving a description of the braid monodromy factorization
of the affine Eyral-Oka curves as well as a theoretical description of their fundamental groups in \autoref{thm-prin}
which allow us to show that the fundamental groups of the projective Eyral-Oka curves are isomorphic in
\autoref{cor-prin}. Finally, in Section~\ref{sec-invariant} we define a new invariant of the embedding of fibered 
curves and use it to produce examples of complement-equivalent arithmetic Zariski pairs in \autoref{thm-Zar-pair}.

\section{Construction of Eyral-Oka curves}\label{sec-construction}

In \cite{Oka-Eyral-fundamental-groups}
M.~Oka and C.~Eyral proposed a candidate for an arithmetic Zariski pair of sextics.
This candidate is the first one formed by curves of torus type, i.e. which can be written as 
$f_2^3+f_3^2=0$, for $f_j$ a homogeneous polynomial in $\CC[x,y,z]$ of degree~$j$.

Eyral-Oka curves are irreducible, they have degree 6, and their singularities are given by: two points of type 
$\mathbb{E}_6$, one $\mathbb{A}_5$, and one $\mathbb{A}_2$. The equisingular stratum of such curves is described 
in~\cite{Oka-Eyral-fundamental-groups}, however for the sake of completeness, an explicit construction of this
space will be provided here. In particular, this realization space has two connected components of dimension 0 up 
to projective transformation. 

To begin with proving the basic properties of these curves, let us fix a sextic curve~$\mathcal{C}: f(x,y,z)=0$ 
with the above set of singularities.

\begin{lema}\label{lema:irr}
The curve~$\mathcal{C}$ is rational and irreducible.
\end{lema}

\begin{proof}
Recall that $\mathbb{E}_6$ singularities have a local equation of the form $x^3+y^4$ for a choice of generators
of the local ring $\cO_{\CC^2,0}$. In particular it is an irreducible singularity whose $\delta$-invariant is 3
and thus it can only be present in an irreducible curve of degree at least~4. Since the total degree of $\cC$ is
6, this implies that both $\mathbb{E}_6$ singularities have to be on the same irreducible component. Again, by a
genus argument, the irreducible component containing both $\mathbb{E}_6$ singularities needs to have degree at
least 5, but then the existence of the irreducible singularity $\mathbb{A}_2$ implies that $\cC$ cannot be a quintic 
and a line (note also that no quintic with two $\mathbb{E}_6$ singularities exists, because
of Bezout's Theorem). Hence, if it exists, it has to be irreducible. Also note that the total $\delta$-invariant of the 
singular locus $2\mathbb{E}_6+\mathbb{A}_5+\mathbb{A}_2$ is 10, which implies that the sextic has to be rational.
\end{proof}

We are going to prove now that $\mathcal{C}$ is of torus type. Using an extension of the 
de~Franchis method~\cite{defranchis-sulle} to rational pencils (see~\cite{ji-libgober-mw,tokunaga-torus-albanese}), 
it follows that $\mathcal{C}$ is of torus type if and only if the cyclotomic polynomial of order 6,
$\varphi_6(t)=t^2-t+1$, divides the Alexander polynomial $\Delta_\cC(t)$ of the curve~$\cC$. Moreover, a torus 
decomposition is unique (up to scalar multiplication) if, in addition, the multiplicity of $\varphi_6(t)$ in 
$\Delta_\cC(t)$ is exactly~1. 
The Alexander polynomial of a curve $V(F)=\{F(x,y,z)=0\}$ was introduced by~Libgober in~\cite{li:82}. 
It can be interpreted as the characteristic polynomial of the monodromy action on the first homology group of 
the cyclic covering of $\mathbb{P}^2\setminus V(F)$ defined as the affine surface $t^d-F(x,y,z)=0$ 
in~$\mathbb{C}^3$. Following the notation in \cite{ea:jag} (see also~\cite{li:82,es:82,lv:90}) and ideas 
coming back from Zariski~\cite{zr:31}, the Alexander polynomial can be computed as follows.

\begin{prop}[{\cite[Proposition~2.10]{ea:jag}}]\label{prop:cok}
Let $V(F)$ be a reduced curve of degree~$d$. All the roots of the Alexander polynomial $\Delta_F(t)$ of~$V(F)$ are 
$d$-th roots of unity. Let $\zeta_d^k:=\exp(\frac{2 i \pi k}{d})$. Then the multiplicity of $\zeta_d^k$ 
as a root of $\Delta_F(t)$ equals the number $d_k+d_{d-k}$, where $d_k$ is the dimension of the cokernel 
of the natural map
$$
\rho_k:H^0(\mathbb{P}^2;\mathcal{O}_{\PP^2}(k-3))\to\bigoplus_{P\in\sing(V(F))} 
\frac{\mathcal{O}_{\mathbb{P}^2,P}}{\mathcal{J}_{P,d,k}},
$$
where $\mathcal{J}_{P,d,k}\subset\mathcal{O}_{\mathbb{P}^2,P}$ is an ideal which depends on the germ of $V(F)$ at 
$P\in\sing(V(F))$ and $\frac{k}{d}$. 
\end{prop}

\begin{obs}
\label{rem-superabundance}
Note that $d_k$ can also be described as
\begin{equation}
\label{eq-superabundance}
d_k=\sum_{P\in \sing(V(F))} \dim \frac{\mathcal{O}_{\mathbb{P}^2,P}}{\mathcal{J}_{P,d,k}}
-{\binom{k-1}{2}}+\dim \ker \rho_k.
\end{equation}
In fact, $\ker \rho_k=H^0(\PP^2;\mathcal{J}_{d,k}(k-3))$ the global sections of an ideal sheaf supported on 
$\sing(V(F))$ whose stalk at $P$ is $\mathcal{J}_{P,d,k}$. Curves in this ideal sheaf will be said to 
\emph{pass through} the ideal $\mathcal{J}_{P,d,k}$ for all $P\in \sing(V(F))$ or simply \emph{pass through}
$\mathcal{J}_{d,k}$.
\end{obs}

We can be more precise in the description of the ideal $\mathcal{J}_{P,d,k}$ by means of an embedded resolution
$\sigma:\hat{X}\to\mathbb{P}^2$ of the point $P$ as a singular point of~$V(F)$.  
Assume for simplicity that $P=[0:0:1]$ and let $E_1^P,\dots,E_n^P$ be the exceptional divisors over~$P$. 
Let $N_i$ be the multiplicity of $\sigma^*(F(x,y,1))$ along $E_i^P$ and let $\nu_i-1$ be the multiplicity of 
$\sigma^*(d x\wedge dy)$ along $E_i^P$. Then,
\begin{equation*}
\mathcal{J}_{P,d,k}:=\left\{h\in\mathcal{O}_{\mathbb{P}^2,P}\left\vert
\text{ the multiplicity of }\sigma^*h\text{ along }E_i\text{ is }>
\left\lfloor\frac{k N_i}{d}\right\rfloor -\nu_i 
\right.\right\}.
\end{equation*}

It is an easy exercise to compute these ideals for the singular points of~$\mathcal{C}$.

\begin{lema}\label{lema:J}
Let $\m_P$ be the maximal ideal of $\mathcal{O}_{\mathbb{P}^2,P}$ and let $\ell_P$ be the local equation of the 
tangent line of $\mathcal{C}$ at $P$. Then, the ideal $\mathcal{J}_{P,6,5}\subset\mathcal{O}_{\mathbb{P}^2,P}$ equals
\begin{enumerate}
\enet{\rm(\arabic{enumi})}
 \item $\m_P$ if $P$ is an $\mathbb{A}_2$-point,
 \item $\langle\ell_P\rangle+\m_P^2$ if $P$ is either an $\mathbb{A}_5$-point or an $\mathbb{E}_6$-point,
\end{enumerate}
whereas $\mathcal{J}_{P,6,1}=\mathcal{O}_{\mathbb{P}^2,P}$ at any singular point~$P$.
\end{lema}

\begin{prop}
The multiplicity $m$ of $\varphi_6(t)$ as a factor of $\Delta_\mathcal{C}(t)$ equals~$1$.
In particular $\mathcal{C}$ admits exactly one torus decomposition.
\end{prop}

\begin{proof}
By \autoref{prop:cok}, $m=d_1+d_5=d_5$ since the target of morphism~$\rho_1$ is trivial. 
On the other hand, using equation~\eqref{eq-superabundance} and \autoref{lema:J} one obtains
$d_5=1+\dim\ker \rho_5$.
Finally, note that $\ker \rho_5=0$, since otherwise a conic curve would \emph{pass through} the ideal 
$\mathcal{J}_{6,5}$, contradicting Bezout's Theorem.
\end{proof}

Therefore the result below follows.

\begin{prop}
\label{prop-torus}
The curve $\cC$ is of torus type and it has a unique toric decomposition.
\qed
\end{prop}

In a torus curve $V(F)$, where $F=f_2^3+f_3^2=0$, the intersection points of the conic $f_2=0$ and the cubic 
$f_3=0$ are singular points of $V(F)$. It is an easy exercise to check that singularities of type 
$\ba_2$, $\ba_5$, and $\EE_6$ can be constructed locally as $u^3+v^2$ where $u=0$ is the germ of a conic and 
$v=0$ is the germ of a cubic as follows:
\begin{enumerate}
\enet{(T\arabic{enumi})}
\item\label{A2} For $\mathbb{A}_2$, the curves $f_2=0$ and $f_3=0$ are smooth and transversal at the point.
\item\label{A5} For $\mathbb{A}_5$, the curve $f_3=0$ is smooth at the point and its intersection number
with $f_2=0$ is~$2$, for instance $(v+u^2)^3+v^2=0$.
\item For $\mathbb{E}_6$, the curve $f_2=0$ is smooth at the point, the curve $f_3=0$ is singular
and their intersection number is~$2$, for instance $u^3+(u^2+v^2)^2=0$.
\end{enumerate}
If $\sing V(F)=V(f_2)\cap V(f_3)$, then $V(F)$ is called a \emph{tame} torus curve, otherwise $V(F)$ is 
\emph{non-tame}.

\begin{lema}\label{lema:pencil}
The curve $\mathcal{C}$ is a non-tame torus curve $\cC=V(f_2^3-f_3^2)$. 

Moreover $V(f_2)$ is a smooth conic and $V(f_3)$, $f_3=\ell \cdot q$ is a reducible cubic where $V(q)$ is a smooth conic 
tangent to $V(f_2)$ only at one point and the line $V(\ell)$ passes through the remaining
two points of intersection of the conics.

In particular, the only non-tame singularity is the $\ba_2$-point.
\end{lema}

\begin{proof}
It follows from the explanation above (and B{\'e}zout's Theorem) that the only possible combination
of singularities at the intersection points of $V(f_2)$ and $V(f_3)$ is
$\ba_5+2\EE_6$; they are the singularities of any generic element of the pencil $F_{\alpha,\beta}=\alpha f_2+\beta f_3$.
Note in particular, that the genus of a generic element of the pencil is~$1$, and its resolution provides
an elliptic fibration.

Since $V(f_3)$ has two double points (at the points of type~$\EE_6$), it must be reducible and the
line $V(\ell)$ joining these two points is one of the components. Let $q:=\frac{f_3}{\ell}$.
Recall that $\cC$ is tangent to $V(f_3)$ at the point of type~$\ba_5$; in fact, it must be
tangent to $V(q)$. Using again B{\'e}zout's Theorem $V(q)$ must be smooth.

Let us resolve the pencil. It is easily seen that it is enough to perform the minimal embedded
resolution of the base points of the pencil. We obtain a map $\tilde{\Phi}:\tilde{X}\to\mathbb{P}^1$
where $\chi(\tilde{X})=14$ and the generic fiber is elliptic. The curve $V(f_3)$ produces
a singular fiber, see \autoref{fig:elliptic}, with four irreducible components: the strict transforms
of the line and the conic, and the first exceptional components~$A_1,B_1$ of blow-ups of the $\EE_6$-points.
We can blow-down the $-1$-rational curves (the strict transforms of the line and the conic) 
in order to obtain a relatively minimal map $\Phi:X\to\mathbb{P}^1$.
The above fiber becomes a Kodaira singular fiber of type $I_2$, while $\cC$ becomes a singular fiber
of type~$II$. 

For the fiber coming from $V(f_2)$, its type changes if $V(f_2)$ is smooth of reducible:
it is of type $\tilde{E}_6$ (smooth) or $\tilde{E}_7$ (irreducible), as it can be seen from \autoref{fig:elliptic}.
An Euler characteristic argument on this elliptic fibration 
shows that
$V(f_2)$ is smooth.
\end{proof}

%
%

After a projective change of coordinates, we can assume that
$P_1=[1:0:0]$, $P_2=[0:0:1]$ (the $\EE_6$-points), $Q=[0:1:0]$ (the $\ba_5$-point), $\ell=y$ and $q=x z-y^2$, where 
$V(\ell)\cap V(f_2)=\{P_1,P_2\}$ and $V(f_2)\cap V(f_3)=\{P_1,P_2,Q\}$.
Note moreover that only the projective automorphism $[x:y:z]\mapsto[z:y:x]$
and the identity globally fix the above points and curves.
The equation of $f_2$ must be:
$$
(x-y)(z-y)-u y (x-2y+z),
$$
for some $u\in\mathbb{C}^*$.

\begin{prop}
Any Eyral-Oka curve $\mathcal{C}$ is projectively equivalent to 
\begin{equation}
\mathcal{C}_a: y^2 (x z-y^2)^2-48 (26 a+45) f_2(x,y,z)^3=0. 
\end{equation}
where $f_2(x,y,z)=(x-y)(z-y)+4(a+2) y (x-2y+z)$ and $a^2=3$. 

Moreover, the curves $\mathcal{C}_+:=\mathcal{C}_{\sqrt{3}}$ and $\mathcal{C}_-:=\mathcal{C}_{-\sqrt{3}}$ 
are not projectively equivalent (in particular, they are not rigidly isotopic).
\end{prop}

\begin{proof}
For a generic value of~$u$, the meromorphic function $\frac{f_2^3}{\ell^2q^2}$ has two  
critical values outside $0,\infty$. Computing a discriminant we find the values of~$u$
for which only one double critical value arises, obtaining the required equation $f_2$.

The computation of the critical value gives the equations in the statement. The result follows from the fact 
that $\mathcal{C}_\pm$ are invariant by $[x:y:z]\mapsto [z:y:x]$.
\end{proof}

\begin{obs}
\label{obs-mult4}
If $L$ denotes the tangent line to $V(f_2)$ (and $V(q)$) at $Q$, then note that $(\mathcal{C}\cdot L)_Q=4$.
This can be computed using the equations but it is also a direct consequence of the construction of $\cC$.
Since $\cC$ at $Q$ has an $\ba_5$ singularity and $L$ is smooth, then using Noether's formula of intersection
$(\mathcal{C}\cdot L)_Q$ can be either 4 or 6. The latter case would imply that 
$(V(f_3)\cdot L)_Q=(V(q)\cdot L)_Q=3$, contradicting B\'ezout's Theorem.
\end{obs}

This construction was already given in~\cite{Oka-Eyral-fundamental-groups}. Let us end
this section with another particular feature of these curves.

Let us perform a change of coordinates such that $L=V(z)$ is the tangent line to $\mathcal{C}$ at $Q$ and the 
normalized affine equation of $\mathcal{C}$ in $x,y$ is symmetric by the transformation $x\mapsto -x$.
One obtains:
\begin{equation}\label{eq:afin}
0=h_a(x,y)=  y^2 \left(y -\frac{\left( x^{2}  -1\right)}{4}\right)^2+\frac{1}{2}
\left(\frac{2 a}{3}y-\frac{2 a+3}{24}  \left(x^{2}  -1\right)\right)^3
\end{equation}
A direct computation shows that $h_a(x,-y+\frac{x^2-1}{4})=h_{-a}(x,y)$,
i.e. these affine curves are equal. 
Let us interpret it in a computation-free way. 

Recall from \autoref{obs-mult4} that $(\cC\cdot L)_{Q}=4$. As in the proof of \autoref{lema:pencil},
we blow up the indeterminacy of the pencil map
$\PP^2\dashrightarrow \PP^1$ defined as $[x:y:z]\mapsto [f_2^3:f_3^2]$, whose fibers are denoted by 
$V(F_{\alpha,\beta})$, for $F_{\alpha,\beta}=\alpha f_{2}^3+\beta f_3^2$.
A picture of these fibration is depicted in \autoref{fig:elliptic}. Most of the exceptional components
of this blow-up are part of fibers. The last components $A_4,B_4$ over the $\EE_6$-points are sections 
while the last component~$E_3$ over the $\ba_5$-point is a $2$-section, that is, the elliptic fibration resctricted to 
this divisor is a double cover of $\PP^1$. It is ramified at the intersections with $E_2$ (in the fiber of 
$V(F_{1,0})$) and the
strict transform of $V(q)$ (in the fiber of 
$V(F_{0,1})$); they have both multiplicity~$2$.

In \autoref{fig:elliptic}, we show also the strict transform of~$L$, the tangent line at the $\ba_5$-point.
One check that this strict transform
becomes a 2-section; one ramification point is the intersection with $E_2$ (in the fiber of 
$V(F_{1,0})$) and the other one 
is the intersection with the strict transform of~$V(\ell)$ (in the fiber of 
$V(F_{0,1})$). 
In particular, there is no more ramification and hence $L$ intersects all other fibers 
$V(F_{\alpha,\beta})$ at two distinct points.

It is clear in \autoref{fig:elliptic} that the combinatorics of $E_4$ and $L$ coincide. In other words, interchanging the roles of $E_3$ and $L$ 
and blowing down accordingly, then one obtains a birational transformation of $\PP^2$ recovering a sextic curve 
$\cC'$ with the same combinatorics and a line $L'$ (the transformation of $E_3$) which is tangent at the $\ba_5$-point. 
Note that this birational transformation exchanges the line $V(\ell)$ (resp. the conic $V(q)$) and the corresponding 
conic $V(q')$ (resp. line $V(\ell')$). In particular, this implies that the transformation cannot be projective. 
Thus, this transformation exchanges curves of types $\cC_+$ and $\cC_-$ and results in the following.

\begin{thm}\label{prop-iso}
The complements $\mathbb{P}^2\setminus\left(L_\pm\cup\mathcal{C}_\pm\right)$ are analytically isomorphic.
\qed
\end{thm}

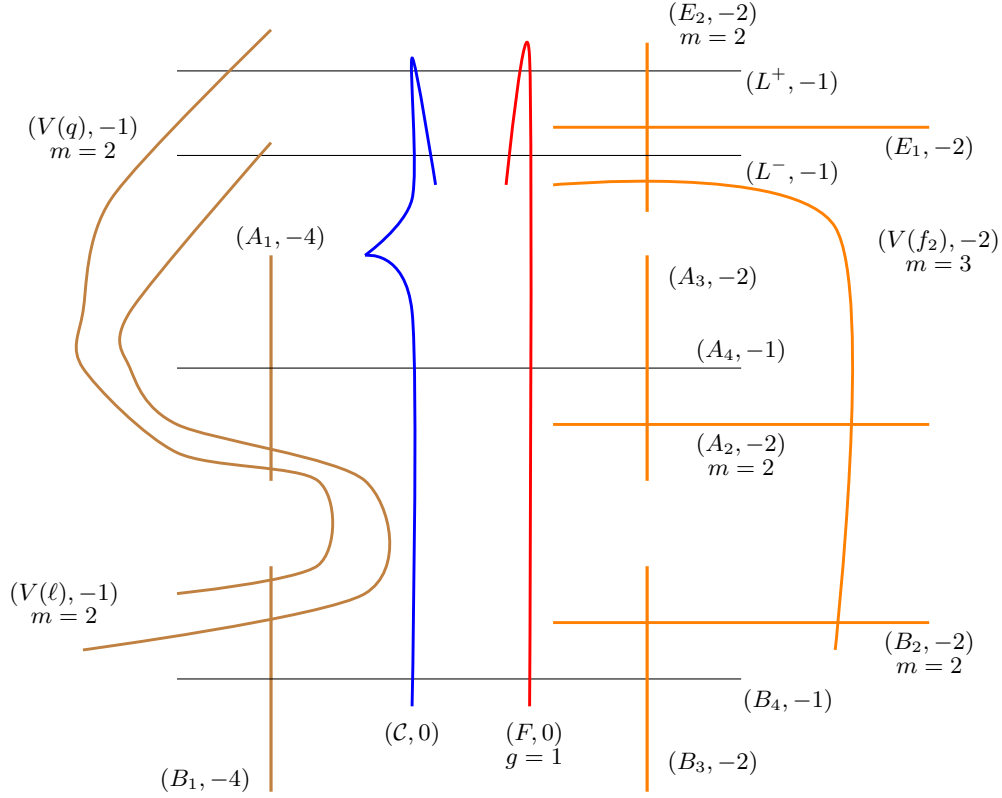
\begin{figure}[ht]
\begin{center}
\iftikziii
\begin{tikzpicture}[xscale=1.25,yscale=.75,vertice/.style={draw,circle,fill,minimum size=0.2cm,inner sep=0}]
\def\evi{
\draw[color=brown,line width=1.2] (-2,-1)--(-2,3);
\draw (-3,1)--(3,1);
\draw[color=orange,line width=1.2] (2,-1)--(2,3);
\draw[color=orange,line width=1.1] (1,0)--(5,0);
}
\def\av{
\draw (-3,1)--(3,1);
\draw[color=orange,line width=1.2] (2,-1.5)--(2,1.5);
\draw[color=orange,line width=1.1] (1,0)--(5,0);
\draw (-3,-.5)--(3,-.5);
}
\def\aii{
\draw (-1,0)--(-1,2);
\draw (-1.25,1)--(.25,1);
\draw (0,0)--(0,2);
}
\evi
\begin{small}
\node at (-1.9,3.3) {$(A_1,-4)$};
\node[below] at (3,0) {$(A_2,-2)$};
\node[below=10pt] at (3,0) {$m=2$};
\node[] at (2.7,2.6) {$(A_3,-2)$};
\node[] at (3,1.3) {$(A_4,-1)$};
\end{small}
\begin{scope}[yshift=-100, yscale=-1]
\evi
\begin{small}
\node at (-2.7,2.8) {$(B_1,-4)$};
\node[below] at (5,0) {$(B_2,-2)$};
\node[below=10pt] at (5,0) {$m=2$};
\node[] at (2.7,2.5) {$(B_3,-2)$};
\node[] at (3.5,1.4) {$(B_4,-1)$};
\end{small}
\end{scope}
\begin{scope}[yshift=150]
\av
\begin{small}
\node[below] at (5,0) {$(E_1,-2)$};
\node[] at (2.7,2.) {$(E_2,-2)$};
\node[below=2pt] at (2.7,2.)  {$m=2$};
\node[below right] at (2.95,1.2) {$(E_3,-1)$};
\node[below right] at (2.95,-.4) {$(L,-1)$};
\end{small}
\end{scope}
\draw[color=orange,line width=1.1] plot [smooth] coordinates {(4,-4) (4,3.5)(1,4.25)};
\draw[color=brown,line width=1.1] plot [smooth] coordinates {(-4,-4) (-1,-3)  (-1,-1)(-3,0)
(-3.5,1) (-3.5,2) (-2,5)};
\draw[color=brown,line width=1.1]  plot [smooth] coordinates {(-3,-3) (-1.5,-2.5)  (-1.5,-1)(-3,-0.5)
(-4,1) (-4,2) (-3.7,4) (-2,7)};
\draw[color=red,line width=1.1] plot [smooth] coordinates {(.75,-5) (.75,6.5)(0.5,4.25)};
\draw[color=blue,line width=1.1] plot [smooth] coordinates {(-.5,-5)(-.5,2) (-1,3) };
\draw[color=blue,line width=1.1] plot [smooth] coordinates {(-1,3) (-.5,4)(-.5,6.5)(-0.25,4.25)};
\begin{small}
\node at (.8,-5.5) {$(F,0)$};
\node[below=2pt] at (.8,-5.5) {$g=1$};
\node at (-.5,-5.5) {$(\mathcal{C},0)$};
\node at (5.1,3.25) {$(V(f_2),-2)$};
\node[below=2pt] at (5.1,3.25) {$m=3$};
\node at (-4,5.25) {$(V(q),-1)$};
\node[below=2pt] at (-4,5.25) {$m=2$};
\node at (-4.2,-3) {$(V(\ell),-1)$};
\node[below=2pt] at (-4.2,-3) {$m=2$};
\end{small}
\end{tikzpicture} 
\else
\includegraphics{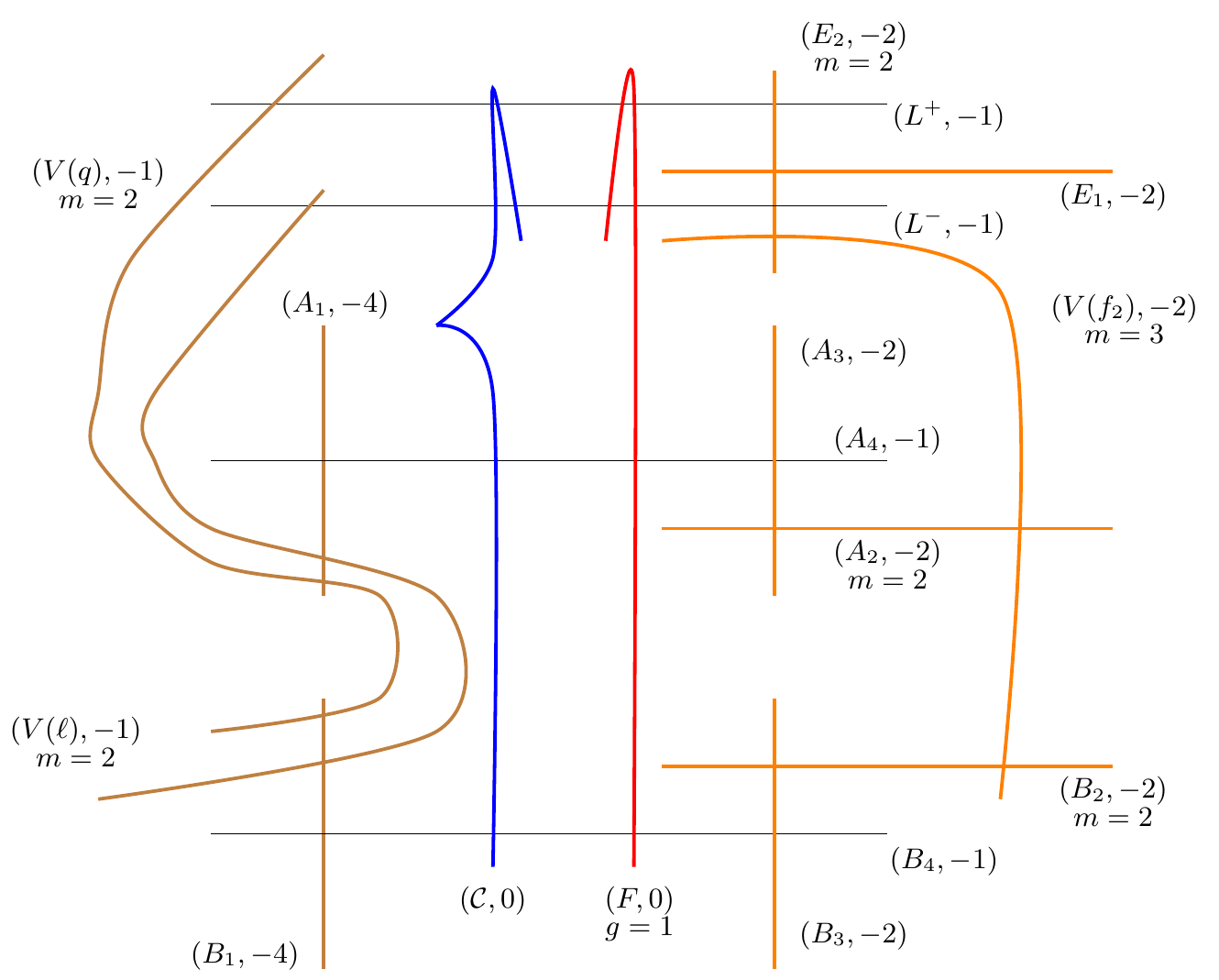}
\fi 
\end{center}
\caption{Elliptic fibration}
\label{fig:elliptic}
\end{figure}

\section{Fundamental group of Eyral-Oka curves}
\label{sec-okaeyral}

The main tool to compute the fundamental group of the complement of a plane curve is the Zariski-van Kampen method. 
In fact, in this method the computation of the fundamental group of $\mathbb{C}^2\setminus\mathcal{C}^{\text{\rm aff}}$
for a suitable affine part of the projective curve~$\mathcal{C}$ is obtained first.
Namely, a line $L$ is chosen (the line at infinity) so that $\mathbb{C}^2\equiv\mathbb{P}^2\setminus L$ is defined
and thus $\mathbb{C}^2\setminus\mathcal{C}^{\text{\rm aff}}\equiv\mathbb{P}^2\setminus\left(L\cup\mathcal{C}\right)$.
Once $\pi_1(\mathbb{C}^2\setminus\mathcal{C}^{\text{\rm aff}})$ is obtained, the fundamental group of 
$\mathbb{P}^2\setminus\mathcal{C}$ can be recovered after factoring out by (the conjugacy class of) a meridian 
of~$L$~\cite{fujita:82}. This is particularly simple if $L\pitchfork\mathcal{C}$, but the argument also follows for 
arbitrary lines. Applying \autoref{prop-iso}, one only needs to compute the fundamental group for one of the affine 
curves, since they are isomorphic. Finally, factoring out by (the conjugacy class of) a meridian of $L$ or $E_3$ 
will make the difference between the groups of the respective curves~$\cC_{\pm}$.

The Zariski-van Kampen method uses a projection $\mathbb{C}^2\to\mathbb{C}$, say the vertical one.
In \autoref{fig:okaeyral}, we have drawn a real picture of the affine curves $\mathcal{C}^{\text{\rm aff}}_{+}$ in 
$\mathbb{P}^2\setminus L_{+}$. For each vertical line, we have also drawn the real part of the complex-conjugate part
as  dotted lines. 

First, we study the situation at infinity.

\subsection{The topology at infinity}
\mbox{}

In order to understand the topology at infinity, let us simplify the construction of the elliptic fibration, 
carried out at the end of the previous section by minimizing the amount of blowing ups and blowing downs as follows.

Let us consider a sequence of blow-ups as in \autoref{fig:blowup}.

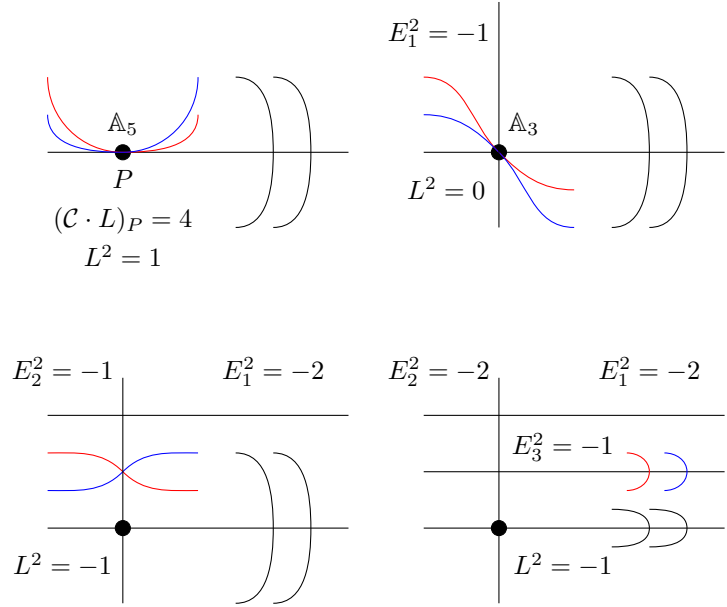
\begin{figure}[ht]
 \centering
\begin{tikzpicture}[vertice/.style={draw,circle,fill,minimum size=0.2cm,inner sep=0}]

\draw (-1,0) -- (3,0);
\node[vertice] (O1) at (0,0) {};
\node[below] at (O1.south)  {$P$};
\draw[color=red]  ($(O1.center)+(-1,1)$) to[out=-90,in=180]  (O1.center) to[out=0,in=-90]
($(O1.center)+(1,1/2)$);
\draw[color=blue]  ($(O1.center)+(-1,1/2)$) to[out=-90,in=180]  (O1.center) to[out=0,in=-90]
($(O1.center)+(1,1)$);
\node[above] at (O1.north)  {$\mathbb{A}_5$};
\node[below] at ($(O1.south)-(0,.5)$)  {$(\mathcal{C}\cdot L)_{P}=4$};
\node[below] at ($(O1.south)-(0,1)$)  {$L^2=1$};
\draw  ($(1.5,1)$) to[out=0,in=90]  (2,0) to[out=-90,in=0] (1.5,-1);
\draw  ($(2,1)$) to[out=0,in=90]  (2.5,0) to[out=-90,in=0](2,-1);

\begin{scope}[xshift=5cm]
\draw (-1,0) -- (3,0);
\node[vertice] (O1) at (0,0) {};
\draw[color=red]  ($(O1.center)+(-1,1)$) to[out=0,in=135]  (O1.center) to[out=-45,in=180]
($(O1.center)+(1,-1/2)$);
\draw[color=blue]  ($(O1.center)+(-1,1/2)$) to[out=0,in=135]  (O1.center) to[out=-45,in=180]
($(O1.center)+(1,-1)$);
\node[above right] at (O1.north)  {$\mathbb{A}_3$};
\node[below] at ($(O1.south)-(.7,.1)$)  {$L^2=0$};
\node[below] at ($(O1.south)-(.8,-2)$)  {$E_1^2=-1$};
\draw  ($(1.5,1)$) to[out=0,in=90]  (2,0) to[out=-90,in=0] (1.5,-1);
\draw  ($(2,1)$) to[out=0,in=90]  (2.5,0) to[out=-90,in=0](2,-1);
\draw (0,-1)--(0,2);
\end{scope}

\begin{scope}[yshift=-5cm]
\draw (-1,0) -- (3,0);
\node[vertice] (O1) at (0,0) {};
\draw[color=red]  (-1,1) to[out=0,in=135]  (0,.75) to[out=-45,in=180]
(1,.5);
\draw[color=blue]  (-1,.5) to[out=0,in=225]  (0,.75) to[out=45,in=180]
(1,1);
\node[below] at ($(O1.south)-(.8,.1)$)  {$L^2=-1$};
\node[below] at ($(O1.south)-(.8,-2.5)$)  {$E_2^2=-1$};
\node[below] at ($(O1.south)+(2,2.5)$)  {$E_1^2=-2$};
\draw  ($(1.5,1)$) to[out=0,in=90]  (2,0) to[out=-90,in=0] (1.5,-1);
\draw  ($(2,1)$) to[out=0,in=90]  (2.5,0) to[out=-90,in=0](2,-1);
\draw (0,-1)--(0,2);
\draw (-1,1.5)--(3,1.5);
\end{scope}

\begin{scope}[xshift=5cm,yshift=-5cm]
\draw (-1,0) -- (3,0);
\node[vertice] (O1) at (0,0) {};
\draw (-1,.75)--(3,.75);
\draw[color=red]  (1.7,1) to[out=0,in=90]  (2,.75) to[out=-90,in=0]
(1.7,.5);
\draw[color=blue]  (2.2,1) to[out=0,in=90]  (2.5,.75) to[out=-90,in=0]
(2.2,.5);
\node[below] at ($(O1.south)-(-.85,.1)$)  {$L^2=-1$};
\node[below] at ($(O1.south)-(.8,-2.5)$)  {$E_2^2=-2$};
\node[below] at ($(O1.south)-(-.85,-1.5)$)  {$E_3^2=-1$};
\node[below] at ($(O1.south)+(2,2.5)$)  {$E_1^2=-2$};
\draw  ($(1.5,.25)$) to[out=0,in=90]  (2,0) to[out=-90,in=0] (1.5,-.25);
\draw  ($(2,.25)$) to[out=0,in=90]  (2.5,0) to[out=-90,in=0](2,-.25);
\draw (0,-1)--(0,2);
\draw (-1,1.5)--(3,1.5);
\end{scope}
\end{tikzpicture}
 \caption{Sequence of blowups}
 \label{fig:blowup}
\end{figure}

\begin{enumerate}
\enet{(B\arabic{enumi})} 
\item The first picture represents a neighborhood of $L$ in $\mathbb{P}^2$. 
\item The second picture is a neighborhood of the total transform of $L$ by the blowing-up of~$Q$. 
Let us denote by $E_1$ the exceptional divisor (this notation will also be used for its strict transforms). 
Note that $E_1\cap L$ is a point of type $\mathbb{A}_3$ in $\mathcal{C}$ which is transversal to both divisors.
\item The third part is a neighborhood of the total transform of the divisor $E_1+L$ by the blow-up of $E_1\cap L$. 
In this case $E_2$ denotes the new exceptional component, which intersects $\cC$ at a nodal point not lying on 
$L\cup E_1$.
\item The fourth picture is obtained by blowing up that nodal point. For convenience, $E_3=L'$ will denote the new 
exceptional component. Note that the divisors $L'$ and $L$ are combinatorially indistinguishable.
\end{enumerate}

\begin{lema}
The sequence of blowing-ups and blowing-downs of \emph{\autoref{fig:blowup}} and{\rm~\ref{fig:ruled}}
converts the projection on vertical lines into the projection $\pi:\Sigma_2\to\mathbb{P}^1$
such that the strict transform of~$\mathcal{C}$ is disjoint with the negative section.

The exceptional components of the blowing-ups of the nodal points at infinity are the strict
transforms of the lines~$L_\pm$.
\end{lema}

\begin{figure}[ht]
\begin{center}
\begin{subfigure}[b]{.45\textwidth}
\begin{tikzpicture}[yscale=-1,vertice/.style={draw,circle,fill,minimum size=0.2cm,inner sep=0}]
\def\ramaderecha{
\node[vertice] (O) at (0,3) {};
\node (A1) at (.5,4) {};
\node (A2) at (1,4.2) {};
\node (A3) at (1.5,4.1) {};
\node(A4) at (2,3.9) {};
\node (A5) at (3,3.8) {};
\draw[line width=1.2] (O.center) to[out=90,in=-155] (A1.center) to[out=25,in=180] (A2.center) to[out=0,in=-215] (A3.center) to[out=-35,in=-210] (A4.center) to[out=-30,in=180] (A5.center) ;
\node (P) at (0,0) {};
\node(B1) at (1,.05) {};
\node[vertice] (B2) at (2,.4) {};
\node (B3) at (3,1) {};
\draw[line width=1.2] (P.center) to[out=0,in=175] (B1.center) to[out=-5,in=180] (B2.center) to[out=0,in=245] (B3.center) ;

\node(C1) at (.5,2) {};
\node(C2) at (1,1.5) {};
\node(C3) at (3,0) {};

\node[right] at (O.east) {$\mathbb{A}_2$};
\node[below right] at (B2) {$\mathbb{E}_6$};

\draw[line width=1.2,dashed] (O.center) to[out=-90,in=95] (C1.center) to[out=-85,in=130] (C2.center)to[out=-50,in=180] (1.5,-.5) to[out=0,in=180] (B2.center) to[out=0,in=180] (C3.center) ;
\draw (2,4.5)--(2,-.5);}
\ramaderecha
\draw (0,4.5)--(0,-.5);
\begin{scope}[xscale=-1]
\ramaderecha
\end{scope}
\end{tikzpicture}
 \caption{Real affine picture of Eyral-Oka's curve}
 \label{fig:okaeyral}
\end{subfigure}
\hspace{3mm}
\begin{subfigure}[b]{.45\textwidth} 
\begin{tikzpicture}[xscale=-1,scale=1.7,vertice/.style={draw,circle,fill,minimum size=0.2cm,inner sep=0}]
\begin{scope}
\node[vertice] (O1) at (0,0) {};
\draw[color=red]  (-1,1) to[out=0,in=135]  (0,.75) to[out=-45,in=180]
(1,.5);
\draw[color=blue]  (-1,.5) to[out=0,in=225]  (0,.75) to[out=45,in=180]
(1,1);
\node[below] at ($(O1.south)-(.45,0.5)$)  {$E_2^2=0$};
\node[below] at ($(O1.south)+(1,1.5)$)  {$E_1^2=-2$};
\draw  ($(-1,.25)$) to[out=0,in=135]  (0,0) to[out=-45,in=0] (1,-.25);
\draw  ($(-1,-.25)$) to[out=0,in=225]  (0,0) to[out=45,in=0](1,.25);
\draw (0,-1)--(0,2);
\draw (-1,1.5)--(1.5,1.5);
\node[vertice] (O1) at (0,.75) {};
\draw[dashed] (.5,-1)--(.5,2);
\end{scope}
\end{tikzpicture}
\caption{Ruled surface near infinity}
\label{fig:ruled}
\end{subfigure}
\end{center}
\caption{}
\end{figure}
%

Let us explain how we have constructed \autoref{fig:okaeyral}.

\begin{enumerate}
\enet{(F\arabic{enumi})} 
\item We compute the discriminant of $h_a(x,y)$ of \eqref{eq:afin} with respect to~$y$. This allows
to check that the~$\mathbb{A}_2$ point is in the line~$x=0$ and the~$\mathbb{E}_6$ point
is in the lines~$x=\pm 1$.
\item We factorize the polynomials $h_a(x_0,y)$ for $x_0=0,\pm 1$ and we obtain which intersection points
are up and down.
\item For $x_0=0,\pm1$, let $y_0$ be the $y$-coordinate of the singular point. For the~$\mathbb{A}_2$ point, 
$y_0=-\frac{a+1}{8}$ we check that the cusp is tangent to the vertical line, and the Puiseux parametrization 
is of the form $y=y_0+y_1 x^{\frac{2}{3}}+\dots$. We obtain that $y_1<0$ and it implies that the real part 
of the complex solutions is bigger than the real solution, near $x=0$.
\item We proceed in a similar way for the $\mathbb{E}_6$ point; in this case, a Puiseux parametrization is
of the form $y=y_0+y_1 x+y_2 x^{\frac{4}{3}}+\dots$, and $y_2<0$.
\item\label{F5} With this data, we draw \autoref{fig:okaeyral}. Note that between the vertical fibers $x=0,1$, 
we have an odd number of crossings. We show later that the actual number is irrelevant for the computations.
\item\label{F6} If we look the situation at $\infty$ in $\Sigma_2$ (\autoref{fig:ruled}), we check that the imaginary branches are still up. 
Hence, from $x=1$ to $\infty$ there is an even number of crossings with the real parts and, as before,
the actual number is irrelevant.
\end{enumerate}

\begin{obs}
The two real branches that go to infinity are the real part of the branches of $\cC_-$
at~$P_-$, while the two conjugate complex branches belong to the branches of $\cC_+$ at~$P_+$.
\end{obs}

\subsection{Strongly real curves and braid monodromy factorization}
\mbox{}

This curve is said to be \emph{strongly real} since it is real, all its affine singularities are real, 
and thus \autoref{fig:okaeyral} contains all the information to compute the braid monodromy 
of~$\mathcal{C}^{\text{aff}}$ and, as a consequence, the fundamental group of its complement.

Let $f(x,y)\in\mathbb{C}[x,y]$ be a monic polynomial in~$y$. The braid monodromy of $f$ with respect to its 
vertical projection is a group homomorphism $\nabla:\mathbb{F}_r\to\mathbb{B}_d$, where $d:=\deg_y f$ and $r$ 
is the number of distinct roots $\{x_1,\dots,x_r\}$ of the discriminant of~$f$ with respect to~$y$ (i.e. the 
number of non-transversal vertical lines to $f=0$). In our case, $r=3$ and $d=4$. In order to calculate $\nabla$,
one starts by considering $x=x_0$ a transversal vertical line and $\{y_1,\dots,y_d\}$ the roots of $f(x_0,y)=0$. 
By the continuity of roots, any closed loop $\gamma$ in $\mathbb{C}$ based at $x_0$ and avoiding the discriminant 
defines a braid based at  $\{y_1,\dots,y_d\}$ and denoted by $\nabla(\gamma)$. Since the vertical projection 
produces a locally trivial fibration outside the discriminant, the construction of the braid only depends on the
homotopy class of the loop $\gamma$. This produces the well-defined morphism~$\nabla$.

Moreover, the morphism $\nabla$ can be used to define an even finer invariant of the curve called the braid
monodromy factorization, via the choice of a special geometrically-based basis of the free group~$\mathbb{F}_r$.
Note that the group $\mathbb{F}_r$ can be identified with $\pi_1(\mathbb{C}\setminus\{x_1,\dots,x_r\};x_0)$ 
and a basis can be chosen by meridians $\gamma_i$ around $x_i$ such that $(\gamma_r\cdot\ldots\cdot\gamma_1)^{-1}$ 
is a meridian around the point at infinity (this is known as a \emph{pseudo-geometric basis}). 
A braid monodromy factorization of $f$ is then given by the $r$-tuple of 
braids~$(\nabla(\gamma_1),\dots,\nabla(\gamma_r))$.

The morphism $\nabla$ is enough to determine the fundamental group of the curve, however a braid monodromy 
factorization is in fact a topological invariant of the embedding of the fibered curve resulting from the union
of the original curve with the preimage of the discriminant (see~\autoref{thm-acc}).

In order to compute a braid monodromy factorization, two important choices are required. First a pseudo-geometric 
basis in $\pi_1(\mathbb{C}\setminus\{x_1,\dots,x_r\};x_0)\equiv\mathbb{F}_r$ and second, an identification
between the braid group based at $\{y_1,\dots,y_d\}$ and the standard Artin braid group~$\mathbb{B}_d$.
This is done with the following choices, see~\cite{acc:01b}.

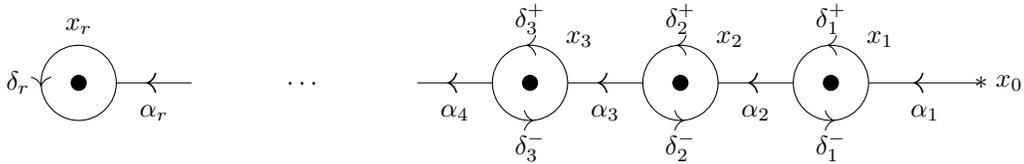
\begin{figure}[ht]
\begin{center}
\iftikziii
\begin{tikzpicture}[scale=2,vertice/.style={draw,circle,fill,minimum size=0.2cm,inner sep=0}]
\tikzset{flecha/.style={decoration={
  markings,
  mark=at position #1 with  {\arrow[scale=1.5]{>}}},postaction={decorate}}}
\def\rd{.25}
\foreach \a in {1,...,7}
{
\coordinate (P-\a) at (10-\a,0);
}
\node[] at (P-1) {$*$};
\node[right=2pt] at (P-1) {$x_0$};
\node at ($.5*(P-5)+.5*(P-6)$) {$\dots$};
\foreach \a in {2,3,4,7}
{
\node[vertice] at (P-\a) {};
\draw (P-\a) circle (\rd);
}
\draw[flecha=0.6] (P-1)-- node [below=5pt] {$\alpha_1$} ($(P-2)+(\rd,0)$);
\foreach \a [evaluate={\b=int(\a+1);}] in {2,3,4}
{
\draw[flecha=.6] ($(P-\a)-(\rd,0)$)-- node [below=5pt] {$\alpha_\a$} ($(P-\b)+(\rd,0)$);
}
\draw[flecha=.6] ($(P-6)-(\rd,0)$)-- node [below=5pt] {$\alpha_r$}($(P-7)+(\rd,0)$);
\foreach \a[evaluate={\b=int(\a-1);}] in {2,3,4}
{
\draw[-{>[scale=1.5]}]  ($(P-\a)+(0,\rd)$)--($(P-\a)+(0,\rd)-.1*(\rd,0)$);
\draw[-{>[scale=1.5]}]  ($(P-\a)-(0,\rd)$)--($(P-\a)-(0,\rd)+.1*(\rd,0)$);
\node[above right =10pt] at ($(P-\a)$) {$x_\b$};
\node[above =15pt] at ($(P-\a)$) {$\delta^+_\b$};
\node[below =15pt] at ($(P-\a)$) {$\delta^-_\b$};
}
\draw[-{>[scale=1.5]}]  ($(P-7)-(\rd,0)$)--($(P-7)-(\rd,0)-.1*(0,\rd)$);
\node[above=15pt] at ($(P-7)$) {$x_r$};
\node[left=15pt] at ($(P-7)$) {$\delta_r$};
\end{tikzpicture}
\else
\includegraphics{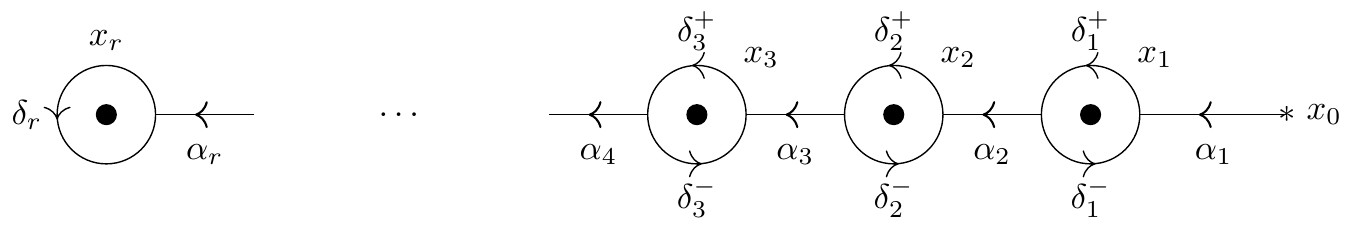}
\fi 
\end{center}
\caption{Pseudo-geometric basis}
\label{fig:geobase}
\end{figure}

\begin{enumerate}
\enet{(C\arabic{enumi})} 
\item For a strongly real curve, a pseudo-geometric basis is chosen as in \autoref{fig:geobase}.
Let 
\[
\underset{1\leq i<r}{\delta_i:=\delta_i^+\cdot\delta_i^-},\quad
\underset{1< i\leq r}{\beta_i:=\prod_{j=1}^{i-1}\alpha_j\cdot\delta_j^+}.
\]
The basis is:
\begin{equation}
\gamma_1:=\alpha_1\cdot\delta_1^+\cdot\delta_1^-\cdot\alpha_1^{-1},\quad
\gamma_i:=\left(\beta_i\cdot\alpha_i\right)\cdot\delta_i^+\cdot\delta_i^-\cdot
\left(\beta_i\cdot\alpha_i\right)^{-1},\ 1<i\leq r.
\end{equation}
Applied to our case, paths $\gamma_1,\gamma_2,\gamma_3$ are required around the points $-1,0,1$ respectively.
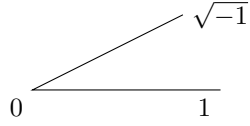
\begin{figure}[ht]
\begin{center}
\begin{tikzpicture}
\coordinate (O) at (0,0);
\coordinate (U) at (2.5,0);
\coordinate (I) at (2,1);
\draw (U)--(O)--(I);
\node[below left] at (O) {$0$};
\node[below left] at (U) {$1$};
\node[right] at (I) {$\sqrt{-1}$};
\end{tikzpicture}
\caption{Complex line}
\label{fig:complex} 
\end{center}
\end{figure}

\item\label{item-identification}
The identification of the braid group on $\{y_1,\dots,y_d\}$ is made using a lexicographic order of the 
roots on their real parts ($\Re y$) and imaginary parts ($\Im y$) such that $\Re y_1\geq\dots\geq\Re y_d$ and 
$\Im y<\Im y'$ whenever $\Re y=\Re y'$. A very useful fact about this canonical construction is that it allows 
one to identify the braids over any path in \autoref{fig:geobase} (whether open or closed) with braids 
in~$\mathbb{B}_d$. 
These conventions can be understood from \autoref{fig:complex}. Namely, projecting the braids onto the real 
line~$\mathbb{R}$, and for complex conjugate numbers we slightly deform the projection such that the positive 
imaginary part number goes to the right and the negative one to the left. In a crossing, the upward strand is the 
one with a smaller imaginary part.

In our case, note that the braid group is~$\mathbb{B}_4$ generated by the Artin system $\sigma_i$, 
$i=1,\dots,3$, the positive half-twist interchanging the $i$-th and $(i+1)$-th strands.

\begin{figure}[ht]
\begin{center}
\iftikziii
\begin{tikzpicture}
\draw ($(3,3)$)-- ($(0,1)$);
\draw[dashed] ($(3,1)$)-- ($(0,3)$);
\node[] at (0,0) {$x=x_2$};
\node[] at (3,0) {$x=x_1$};
\node[] at (8,1) {$x=x_1$};
\node[] at (8,3) {$x=x_2$};
\draw[line width=1.2] (2.5,.5) --(2.5,3.5);
\draw[line width=1.2] (.5,.5) --(.5,3.5);
\draw[-{[scale=1.5]>}] (2.4,.7)--(.6,.7);
\draw[-{[scale=1.5]>}] (8,1.2)--(8,2.9);
\node at (5,.5) {$-$};
\node at (6,.5) {$+$};
\node at (7,.5) {$\mathbb{R}$};
\node at (6,3.5) {$-$};
\node at (7,3.5) {$+$};
\node at (5,3.5) {$\mathbb{R}$};
\braid[xscale=-1,rotate=180,height=.75cm]  at (5,-1) s_2^-1 s_1;
\end{tikzpicture}
\else
\includegraphics{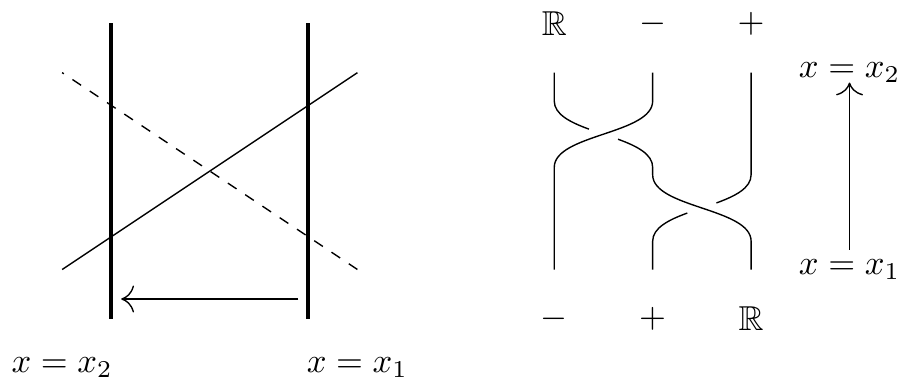}
\fi
\caption{Crossing of a real branch with a couple of complex conjugate branches}
\label{fig:crossing} 
\end{center}
\end{figure}
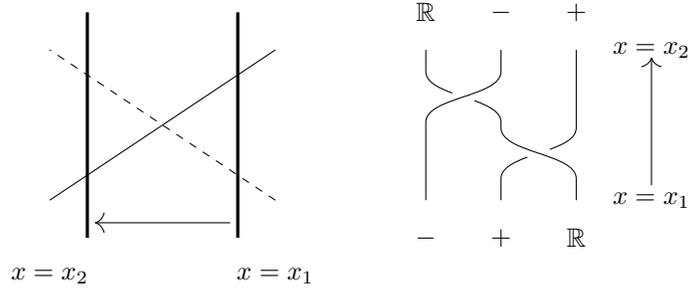

\item 
Given a strongly real curve one can draw its real picture. This real picture might be missing complex conjugate 
branches. For those, one can draw their real parts as shown in \autoref{fig:okaeyral} with dashed curves.
This picture should pass the \emph{vertical line test}, that is, each vertical line should intersect the picture
in $d$ points counted appropriately, that is, solid line intersections count as one whereas dashed line intersections
count as two.

At this point, the braids can be easily recovered as long as the dashed lines have no intersections as follows:
\begin{itemize}
 \item 
 At intersections of solid lines one has a singular point. The local braid over $\delta^+$ and $\delta^-$ can be 
obtained via the Puiseux pairs of the singularity. 
 \item 
 At an intersection of a solid and a dashed line as in \autoref{fig:crossing}, the local braid on three strands 
 $\sigma_1^{-1}\cdot\sigma_2$ is obtained as a lifting of the open path $\alpha$ crossing the intersection, where 
 the generators $\sigma_i$ are chosen locally and according the identification given in~\ref{item-identification}. 
 In the reversed situation (that is, when the solid and dashed lines are exchanged), the inverse braid is obtained.
 This justifies the assertions in \ref{F5} and~\ref{F6}. 
\end{itemize}
In our case the following braids in $\mathbb{B}_4$ are obtained:
$$
\alpha_1\mapsto 1,\quad \alpha_2\mapsto \sigma_2^{-1}\cdot\sigma_1
,\quad \alpha_3\mapsto \sigma_1^{-1}\cdot\sigma_2.
$$
\end{enumerate}

In order to finish the computation of the braid monodromy factorization of $\mathcal{C}^{\text{aff}}$,
we  need to compute the braids associated to $\delta_i^\pm$, $i=1,2$, and $\delta_3$.
Next lemma provides the key tools.

\begin{lema}
Let $f(x,y)=y^3-x$; following the above conventions,
the braid in $\mathbb{B}_3$ obtained from the path $\alpha: t\mapsto x=\exp(2 \sqrt{-1}\pi t)$, $t\in[0,1]$,
equals $\sigma_2\cdot\sigma_1$. 
For $g(x,y)=y^3+x$, the braid associated with $\alpha$ equals $\sigma_1\cdot\sigma_2$.
\end{lema}

\begin{proof}
Note that for $x=1$, the values of the roots of the $y$-polynomial $f(1,y)$ are $1,\zeta,\bar{\zeta}$, 
for $\zeta:=\exp\left(\frac{2\sqrt{-1}\pi}{3}\right)$, and thus the associated braid is nothing but the 
rotation of angle $\frac{2}{3}\pi$.

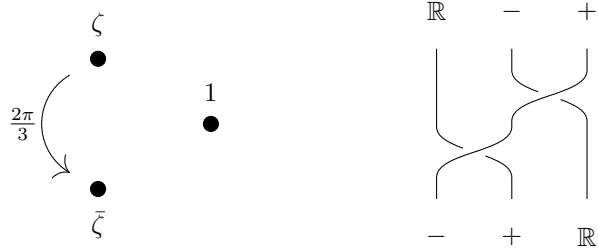
\begin{figure}[ht]
\begin{center}
\iftikziii
\begin{tikzpicture}[vertice/.style={draw,circle,fill,minimum size=0.2cm,inner sep=0}]
\node[vertice] at (0:1) {};
\node[above=5pt] at (0:1) {$1$};
\node[vertice] at (120:1) {};
\node[above=5pt] at (120:1) {$\zeta$};
\node[vertice] at (-120:1) {};
\node[below=5pt] at (-120:1) {$\bar{\zeta}$};
\draw[-{[scale=2]>},xshift=-.5cm] (120:.75) arc(120:240:.75);
\node at (-1.5,0) {$\frac{2\pi}{3}$};
\node at (4,-1.5) {$-$};
\node at (5,-1.5) {$+$};
\node at (6,-1.5) {$\mathbb{R}$};
\node at (5,1.5) {$-$};
\node at (6,1.5) {$+$};
\node at (4,1.5) {$\mathbb{R}$};
\braid[xscale=-1,rotate=180,height=.75cm]  at (4,1) s_1 s_2;
\end{tikzpicture}
\else
\includegraphics{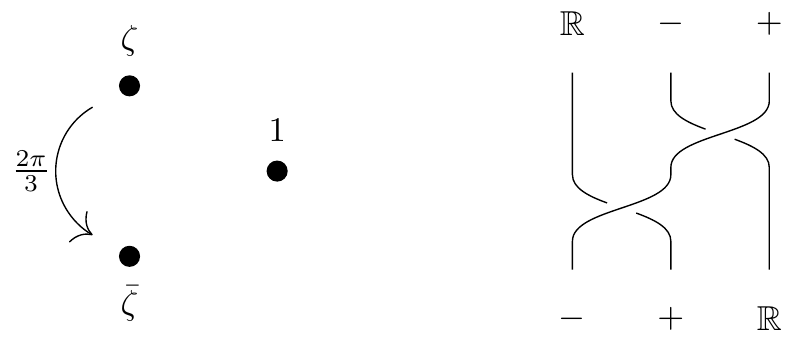}
\fi
\caption{Braid for $y^3=x$.}
\label{fig:tercio} 
\end{center}
\end{figure}

The result follows from the identification described in \ref{item-identification} 
and \autoref{fig:tercio}.
\end{proof}

Applying this in our situation one obtains (see~\autoref{fig:okaeyral}):
\begin{equation}
\delta_1^\pm\mapsto(\sigma_1\cdot\sigma_2)^2\quad
\Longrightarrow \delta_1,\delta_3\mapsto(\sigma_1\cdot\sigma_2)^4
\qquad
\delta_2^\pm\mapsto\sigma_2\cdot\sigma_3\quad
\Longrightarrow \delta_2\mapsto(\sigma_2\cdot\sigma_3)^2.
\end{equation}

Combining all the braids obtained above, one can give the monodromy factorization.

\begin{prop}
\label{prop-braid-monod}
The braid monodromy factorization of $\mathcal{C}^{\text{\rm aff}}$ is
$(\tau_1,\tau_2,\tau_3)$ where:
\begin{align}
\tau_1:=&(\sigma_1\cdot\sigma_2)^4,\nonumber\\
\tau_2:=&(\sigma_1\cdot\sigma_2\cdot\sigma_1^2)\cdot(\sigma_2\cdot\sigma_3)^2
\cdot(\sigma_1\cdot\sigma_2\cdot\sigma_1^2)^{-1},\label{eq:bm}\\
\tau_3:=&(\sigma_2\cdot\sigma_1^2\cdot\sigma_2\cdot\sigma_3)
\cdot(\sigma_2\cdot\sigma_1)^4\cdot(\sigma_2\cdot\sigma_1^2\cdot\sigma_2\cdot\sigma_3)^{-1}.\nonumber
\end{align} 
\end{prop}

\begin{obs}
Note that the closure of $\mathcal{C}^{\text{\rm aff}}$ in the ruled surface $\Sigma_2$ is disjoint from 
the negative section $E_1$. As stated in \cite[Lemma~2.1]{khku:04}, the product of all braids (associated to 
paths whose product in the complement of the discriminant in $\mathbb{P}^1$ is trivial) equals $(\Delta^2)^2$. 
Hence $\Delta^4\cdot(\tau_3\cdot\tau_2\cdot\tau_1)^{-1}$ is the braid associated to two disjoint nodes, 
see~\autoref{fig:ruled}. The following equality is a straightforward exercise:
$$
(\sigma_1^2\cdot\sigma_3^2)\cdot(\tau_3\cdot\tau_2\cdot\tau_1)=\Delta^4.
$$
\end{obs}

\subsection{A presentation of the fundamental group}
\mbox{}

Our next step will be to compute $G:=\pi_1(\mathbb{C}^2\setminus\mathcal{C}^{\text{aff}})$.
The main tool towards this, as mentioned before, entails considering a braid monodromy factorization and its
action on a free group. Before stating Zariski-van Kampen's Theorem  precisely, let us recall this natural 
right action of $\mathbb{B}_d$ on $\mathbb{F}_d$ with basis $g_1,\dots,g_d$ which will be denoted by 
$g^\sigma$ for a braid $\sigma \in \mathbb{B}_d$ and an element $g\in\mathbb{F}_d$. It is enough to describe
it for $g_i$ a system of generators in $\mathbb{F}_d$ and $\sigma_j$ an Artin system of $\mathbb{B}_d$:
\begin{equation}
g_i^{\sigma_j}:=
\begin{cases}
g_{i+1}&\text{ if }i=j\\
g_{i}\cdot g_{i-1}\cdot g_{i}^{-1}&\text{ if }i=j+1\\
g_i&\text{ otherwise}. 
\end{cases}
\end{equation}

The following is the celebrated Zariski-van Kampen Theorem, which allows for a presentation of the 
fundamental group of an affine curve from a given braid monodromy factorization.

\begin{thm}[Zariski-van Kampen Theorem]
If $(\tau_1,\dots,\tau_r)\in\mathbb{B}_d^r$ is a braid monodromy factorization of an affine
curve~$\cC$, then:
$$
\pi_1(\mathbb{C}^2\setminus\cC^{\text{aff}})=
\left\langle
g_1,\dots,g_d
\left|
\right. g_i=g_i^{\tau_j},\quad 1\leq j\leq r,\quad 1\leq i<d
\right\rangle.
$$
If $\tau_i=\alpha_i^{-1}\cdot\beta_i\cdot\alpha_i$, and $\beta_i$ is a (usually positive)
braid involving strands $k_i+1,\dots,k_i+m_i$, then:
\begin{equation}\label{eq:zvk}
\pi_1(\mathbb{C}^2\setminus\cC^{\text{aff}})=
\left\langle
g_1,\dots,g_d
\left|
\right. g_{i}^{\alpha_j}=(g_{i}^{\beta_j})^{\alpha_j},\quad 1\leq j\leq r,\quad k_j< i<k_j+m_j
\right\rangle.
\end{equation}
\end{thm}

\begin{obs}
\label{rem-alphabeta}
In our case, $r=3$, $d=4$, $k_1=k_3=0$, $m_1=m_3=2$, $k_2=1$, $m_2=1$, and
\begin{align}\label{eq:conj}
\alpha_1&=1&\beta_1&=(\sigma_1\cdot\sigma_2)^4\nonumber\\
\alpha_2&=(\sigma_1\cdot\sigma_2\cdot\sigma_1^2)^{-1}&\beta_2&=(\sigma_2\cdot\sigma_3)^2\\
\alpha_3&=(\sigma_2\cdot\sigma_1^2\cdot\sigma_2\cdot\sigma_3)^{-1}&
\beta_3&=(\sigma_2\cdot\sigma_1)^4\nonumber. 
\end{align}
\end{obs}

The previous sections where a braid monodromy factorization allow us to give a presentation of the 
fundamental group of an affine curve complement.

\begin{cor}
Let $\cC^{\text{aff}}$ be the affine Eyral-Oka curve as described at the
beginning of{\rm~\S\ref{sec-okaeyral}},
consider a braid monodromy factorization as
described in {\rm~\autoref{prop-braid-monod}}
and{\rm~\eqref{eq:conj}}.
Then the group $G=\pi_1(\mathbb{C}^2\setminus\cC^{\text{aff}})$ admits a presentation
as
\begin{equation}
\label{eq-presentation}
\langle g_1,\dots,g_4: g_1^{\beta_1}=g_1, g_2^{\beta_1}=g_2,
g_2^{\alpha_2}=(g_2^{\beta_2})^{\alpha_2},
g_3^{\alpha_2}=(g_3^{\beta_2})^{\alpha_2},
g_1^{\alpha_3}=(g_1^{\beta_3})^{\alpha_3},
g_2^{\alpha_3}=(g_2^{\beta_3})^{\alpha_3}\rangle.
\end{equation}
\end{cor} 

Presentation~\eqref{eq-presentation} contains $6$ relations for a total length of~$40$. 
For the sake of clarity, instead of showing an explicit presentation, we will describe the group $G$
in a more theoretical way that would allow to understand its structure. The following description 
characterizes the group $G$ completely.

\begin{thm}\label{thm-prin}
The fundamental group~$G$ can be described as follows:
\begin{enumerate}
\enet{\rm(\arabic{enumi})}
 \item Its derived subgroup $G'\subset G$ can be decomposed as a semidirect product $K\rtimes V$,
where:
\begin{enumerate}
\eneti{\rm(\alph{enumii})}
\item The subgroup $V$ is the Klein group $\langle a,b\mid a^2=b^2=a\cdot b\cdot a^{-1}\cdot b^{-1}=1\rangle$.
\item The subgroup $K$ is the direct product of a rank-$2$ free group and a cyclic group order~$2$ with presentation
$$
\langle x,y,w\mid w^2=x\cdot w\cdot x^{-1}\cdot w=y\cdot w\cdot y^{-1}\cdot w=1\rangle.
$$  
\item The action of $V$ on $K$ is given by:
\begin{equation*}
x^a =x,\quad y^a =y\cdot w,\quad w^a = w,\quad x^b =x\cdot w,\quad y^b =y, \quad w^b =w.
\end{equation*}
In particular, $w$ is central in~$G'$.
\end{enumerate}
\item\label{item-meridian}
There exists a meridian $g$ of $\mathcal{C}^{\text{\rm aff}}$ such that $G=G'\rtimes\mathbb{Z}$,
where $\mathbb{Z}$ is identified as $\langle g\mid -\rangle$ and the action is defined by:
\begin{equation*}
g\cdot x\cdot g^{-1} =y^{-1},\quad g\cdot y\cdot g^{-1} =y\cdot x\cdot b,
\quad g\cdot w\cdot g^{-1} =w,\quad g\cdot a\cdot g^{-1} =b,\quad g\cdot b\cdot g^{-1} =a\cdot b.
\end{equation*}
\item There is a central element~$z$ such that $z\cdot g^6=[y,x]$.
The center of $G$ is generated by $z,w$.
\item There is an automorphism of $G$ sending $z$ to $z\cdot w$.
\end{enumerate}
\end{thm}

\begin{proof}
A presentation of this sort can be obtained using \texttt{Sagemath}~\cite{sage66}
(which contains \texttt{GAP4}~\cite{GAP4} as main engine for group theory).
We indicate the steps of the proof:
\begin{enumerate}
\enet{(G\arabic{enumi})}
\item The original presentation~\eqref{eq-presentation} (with four generators and six relations) can be simplified
to have only two generators and four relations, both generators being meridians of~$\mathcal{C}^{\text{\rm aff}}$.
Any of such meridians can play the role of~$g$ in~\ref{item-meridian}.

\item From the simplified presentation above, one can find a central element~$z\in G$ whose image by the standard
abelianization morphism is $-6$. Recall the abelianization of $G$ is $\ZZ$. Moreover, the abelianization can be 
fixed by setting the image of any meridian to be~1, this is what we call the standard abelianization.

Since $z$ is central, note that $G'\cong (G/\langle z\rangle)'$,  see e.g.~\cite{deg:2010_e8}. Since the latter derived group 
is of index~$6$ in $G/\langle z\rangle$, 
we can apply Reidemeister-Schereier method to find
a finite presentation of~$G'$ with $5$ generators~$x,y,a,b,w\in G'$.

\item From the previous steps it is a tedious computation to verify the structure of $G'$ indicated in the statement
as well to check the conjugation action of~$g$. In particular, that $w$ is central in $G'$ and $z=g^{-6}\cdot[y,x]$.

\item Note that the center of $G$ is the group generated by $z$ and $w$ and that $z,z\cdot w$ are the only central 
elements which are sent to $-6$ by the standard abelianization. Moreover, it is straightforward to prove that
$$
g\mapsto g,\quad x\mapsto y^{-1}\cdot a,\quad y\mapsto y\cdot x\cdot w,\quad a\mapsto b,\quad 
b\mapsto a\cdot b,\quad w\mapsto w
$$
defines an automorphism of $G$ such that $z\mapsto z\cdot w$.
\end{enumerate}
\end{proof}

Going back to the discussion about the topology at infinity, one can detect the meridians of the tangent line 
$L$ and the meridian corresponding to the exceptional divisor $E_3=L'$. These are required to recuperate the
original fundamental groups of the projective Eyral-Oka sextics.

\begin{cor}\label{cor-prin}
The central element $z$ is a meridian of $L_+=L$ while $z\cdot w$ is a meridian of $L_-=L'$.
In particular, the groups $\pi_1(\mathbb{P}^2\setminus\cC_\pm)$ are isomorphic.
\end{cor}

\begin{proof}
Starting from the Zariski-van Kampen presentation~\eqref{eq-presentation}, and using the blow-up blow-down 
process described in \autoref{fig:blowup}, it is straightforward that a meridian of $L=L_+$ (resp. $E_3=L'=L_-$)
is given by $e^2\cdot(g_2\cdot g_1)$ (resp. $e^2\cdot(g_4\cdot g_3)$), where $e=(g_4\cdot\ldots\cdot g_1)^{-1}$.
The result follows from tracing these meridians along the steps described in \autoref{thm-prin}.
\end{proof}

\autoref{cor-prin} answers negatively a question in~\cite{Oka-Eyral-fundamental-groups}.
In the following section this curve will be used to construct arithmetic Zariski pairs that are complement equivalent.

\section{Zariski pairs and braid monodromy factorizations}
\label{sec-invariant}

In \autoref{cor-prin}, we have proved that the fundamental groups of the Eyral-Oka curves are isomorphic, and hence
this invariant cannot be used to decide if these two curves, which are not rigidly equivalent, form an arithmetic 
Zariski pair.

Degtyarev~\cite{degt:08} proved that any two non-rigidly equivalent equisingular sextic curves with simple singularities 
cannot have \emph{regularly homeomorphic} embeddings, where a regular homeomorphism is a homeomophism that is 
holomorphic at the singular points.

In particular, by Degtyarev's result, Eyral-Oka curves are close to being an arithmetic Zariski pair.
Shimada was able to refine Degtyarev's arguments in~\cite{shi:08,shi:09} and developed an $N_C$-invariant that 
was able to exhibit that some of these candidates to Zariski pairs were in fact so. Unfortunately, the $N_C$-invariant 
coincides for the Galois-conjugate projective Eyral-Oka curves.

As we showed in \S\ref{sec-construction}, the curves $\mathcal{C}_\pm\cup L_\pm$ have homeomorphic complements
(even more, analytically isomorphic) via the birational morphism shown in~\autoref{fig:blowup}. The Cremona 
transformation that connects both complements is not a homeomorphism of the pairs 
$(\mathbb{P}^2,\mathcal{C}_\pm\cup L_\pm)$, so these curves are candidates to be complement equivalent Zariski pairs.

We are not able to decide on that problem, but we are apparently more succesful when adding more lines to the 
original curves $\cC_{\pm}\cup L_{\pm}$. More precisely, the curve 
$\mathcal{C}_\pm\cup L_\pm\cup L_{\pm}^2\cup L_{\pm}^{6,1}\cup L_{\pm}^{6,2}$ can be proved to be a
complement-equivalent arithmetic Zariski pair. Note that the extra lines correspond to the preimage of the 
discriminant in~$\PP^1$ (see~\autoref{fig:okaeyral}).

\subsection{Fibered curves and braid monodromy factorization}
\mbox{}

These curves are called \emph{fibered} (see~\cite{acc:01b}) since their complements induce a locally trivial fibration
on a finitely punctured $\PP^1$ (the complement of the discriminant). A fibered curve has a \emph{horizontal part} (the 
curve that intersects the generic fibers in a finite number of points) and a \emph{vertical part} (the preimage of the 
discriminant). As mentioned above, the braid monodromy of its horizontal part is a topological invariant of a 
fibered curve. 

Let us recall this result. Consider $\cC\subset\mathbb{P}^2$ a projective curve, $L\subset\mathbb{P}^2$ be a line
and $P\in L$. Let us assume that, if $P\in\cC$, then $L$ is the tangent cone of $\cC$. Consider $L_1,\dots,L_r$ the 
lines in the pencil through~$P$ (besides~$L$) which are non-transversal to~$\cC$.
The curve~$\cC^\varphi:=\cC\cup L\cup\bigcup_{j=1}^r L_j$ is the \emph{fibered curve} associated with $(\cC,L,P)$.

Consider now the braid monodromy factorization $(\tau_1,\dots,\tau_r)\in\mathbb{B}_d^r$ of the affine curve 
$\cC^{\text{aff}}:=\cC\setminus L$, with respect to the projection based at~$P$, where $d$ is the difference between 
$\deg\cC$ and the multiplicity of~$\cC$ at~$P$. We are ready to state the result.

\begin{thm}[{\cite[Theorem~1]{acc:01a}}]\label{thm-acc} Let us suppose the existence of a homeomorphism 
$\Phi:(\mathbb{P}^2,\cC_1^\varphi)\to(\mathbb{P}^2,\cC_2^\varphi)$ such that:
\begin{enumerate}
\enet{\rm(\arabic{enumi})}
\item The homeomorphism is orientation preserving on $\mathbb{P}^2$ and on the curves.
\item $\Phi(P_1)=P_2$, $\Phi(L_1)=L_2$.
\end{enumerate}
Then, the two triples $(\cC_1,L_1,P_1)$
and $(\cC_2,L_2,P_2)$ have equivalent braid monodromies,
\end{thm}

To understand the statement, let us recall the notion of equivalence of braid monodromies.
Let $(\tau_1,\dots,\tau_r)\in\mathbb{B}_d^r$ be a braid monodromy factorization; for its construction
we have identified the Artin braid group $\mathbb{B}_d$ with the braid group based at some specific
$d$~points of~$\mathbb{C}$; two such identifications differ by conjugation, i.e.,
$$
(\tau_1,\dots,\tau_r)\sim(\tau_1^\tau,\dots,\tau_r^\tau),\quad\forall\tau\in\mathbb{B}_d.
$$
There is a second choice, the choice of a pseudo-geometric basis in~$\mathbb{F}_r$. Two such
bases differ but what is called a \emph{Hurwitz move}. The Hurwitz action of $\mathbb{B}_r$
on $G^r$ (where $G$ is an arbitrary group) is defined as follows. Let us denote by $s_1,\dots,s_{r-1}$
the Artin generators of $\mathbb{B}_r$ (we replace $\sigma$ by~$s$ to avoid confusion when
$G$ is a braid group). Then:
$$
(g_1,\dots,g_r)^{s_i}\mapsto(g_1,\dots,g_{i-1},g_{i+1},g_{i+1}\cdot g_i\cdot g_{i+1}^{-1},\dots,g_{i+2},\dots,g_r).
$$

\begin{dfn}
Two braid monodromies in $\mathbb{B}_d^r$ are \emph{equivalent} if they belong to the same orbit
by the action of~$\mathbb{B}_r\times\mathbb{B}_d$ described above.
\end{dfn}

Note that it is hopeless to apply directly \autoref{thm-acc} to our case: the braid monodromies are equal!
In~\cite{accm}, we refined \autoref{thm-acc} to work with \emph{ordered line arrangements}:
the classical braid groups were replaced everywhere by pure braid groups. We are going to state now
an intermediate refinement of \autoref{thm-acc}.

Let us think about our case. If we color in a different way the two first strands and the two
last strands, we take into account, that the first ones are the branches of the node in~$\Sigma_2$
which provides~$L_+$, while the last ones provide~$L_-$. Let us set that $(\tau_1,\tau_2,\tau_3)$
is the braid monodromy factorization for $\mathcal{C}_+$ with this coloring. To compare both
curves, the braid monodromy factorization of $\mathcal{C}_-$ would have the strands associated
to $L_-$ in the first place; this is accomplished, considering:
$$
(\tilde{\tau}_1,\tilde{\tau}_2,\tilde{\tau}_3):=
(\tau_1^\tau,\tau_2^\tau,\tau_3^\tau),\qquad \tau=(\sigma_2\cdot\sigma_3\cdot\sigma_1)^2
$$
since the braid $\tau$ exchanges the two pairs of strands.

\begin{dfn}
Let $A$ be a partition of the set $\{1,\dots,n\}$. The \emph{braid group} $\mathbb{B}(A)$ relative to $A$
is the subgroup of~$\mathbb{B}_n$ consisting of the braids that respect the given partition.
\end{dfn}

\begin{obs}
For instance, note that both the total and the discrete partition provide recognizable groups:
$\mathbb{B}_n=\mathbb{B}(\{\{1,\dots,n\}\})$ whereas $\mathbb{B}(\{\{1\},\dots,\{n\}\})$ provides the 
pure braid group.
\end{obs}

The proof of the following result follows along the same lines as that of~\cite[Theorem~1]{acc:01a}.

\begin{thm}\label{thm-ref} 
Let $A_d,A_r$ be partitions of $\{1,\dots,d\}$ and  $\{1,\dots,r\}$, respectively, such that $A_d$ induces partitions 
on $L_i\cap\cC_i$. Assume that there exists a homeomorphism 
$\Phi:(\mathbb{P}^2,\cC_1^\varphi)\to(\mathbb{P}^2,\cC_2^\varphi)$ satisfying the hypotheses in 
{\rm \autoref{thm-acc}}, and also satisfying:
\begin{enumerate}
\enet{\rm(\arabic{enumi})}
\item The blocks of lines through~$P_1,P_2$ associated to the partition are respected.
\item The partitions on $L_i\cap\cC_i$ are respected.
\end{enumerate}
Then, the triples $(\cC_1,L_1,P_1)$ and $(\cC_2,L_2,P_2)$ have braid monodromy factorizations
$(\tau_1^j,\dots,\tau_r^j)\in\mathbb{B}_d^r$, $j=1,2$ (respecting
the above partitions) which are equivalent
by the action of $\mathbb{B}(A_d)\times\mathbb{B}(A_r)$.
\end{thm}

\begin{thm}
\label{thm-Zar-pair}
There is no homeomorphism
$$
(\mathbb{P}^2,\mathcal{C}_+\cup L_+\cup L_{+}^2\cup L_{+}^{6,1}\cup L_{+}^{6,2})\to
(\mathbb{P}^2,\mathcal{C}_-\cup L_-\cup L_{-}^2\cup L_{-}^{6,1}\cup L_{-}^{6,2}).
$$
\end{thm}

\begin{proof}
Let us assume that such a homeomorphism $\Phi$ exists. From the topological properties
of~$\mathbb{P}^2$, it must respect the orientation of~$\mathbb{P}^2$. The intersection form
in~$\mathbb{P}^2$ implies that either respect all the orientations on the curves, or reversed
all of them. Since the equations are real, in the latter case we can compose~$\Phi$ with
complex conjugation, and we may assume that $\Phi$ respects all the orientations.

By the local topology of the curves, $\Phi(\mathcal{C}_+)=\mathcal{C}_-$; it is also easy
to check that $\Phi(L_+)=L_-$ and $\Phi(L_+^{2})=L_-^{2}$ (the lines joining the points
of type $\mathbb{A}_5$ and $\mathbb{A}_2$). In particular, $\Phi(P_+)=P_-$. Note also
that $\Phi(L_{+}^{6,1}\cup L_{+}^{6,2})=L_{-}^{6,1}\cup L_{-}^{6,2}$ (the lines 
joining the $\mathbb{A}_5$ point and the two $\mathbb{E}_6$ points).

Moreover, the homeomorphism must respect the two branches of the $\mathbb{A}_5$ point
and, hence, the two other points in~$\mathcal{C}_{\pm}\cap L_\pm$ (globally). 

Let us consider the partition $A_4^d=\{\{1,2\},\{3,4\}\}$ for the strands of the braids.
In the base, we consider the partition $A_3^r=\{\{1,3\},\{2\}\}$. 
Then, from \autoref{thm-ref}, the braid monodromies
$T:=(\tau_1,\tau_2,\tau_3)$ and $\tilde{T}:=(\tilde{\tau}_1,\tilde{\tau}_2,\tilde{\tau}_3)$
are equivalent under the action of $\mathbb{B}(A_4^d)\times\mathbb{B}(A_3^r)$.

We are going to show that this does not happen and, in particular, the expected
homeomorphism does not exist. 

There is no algorithm ensuring that two braid monodromy factorizations are equivalent.
In order to look for necessary conditions, we consider a finite representation
$\varphi:\mathbb{B}_4\to F$, where $F$ is a finite group. We need to check
if $\varphi(T)$ and $\varphi(\tilde{T})$ are equivalent under the action
of $\varphi(\mathbb{B}(A_4^d))\times\mathbb{B}(A_3^r)$. Since the orbits
are finite, this approach should lead to an answer.

Let us denote $F_A:=\varphi(\mathbb{B}(A_4^d))$; let $\hat{F}:=F^3/F_A$, i.e., the quotient
of the cartesian product $F^3$ under the diagonal conjugation action of $F_A$. The 
group $\mathbb{B}(A_3^r)$ acts by Hurwitz moves on it. We want to check if the classes
$[T],[\tilde{T}]\in\hat{F}$ are in the same orbit under this action. Note that in general,
this can be computationally expensive. 

There is a natural way to obtain representations of the braid group. Consider the
reduced Burau representation $\phi:\mathbb{B}_4\to\GL(3,\mathbb{Z}[t^{\pm 1}])$.
Let $R$ be either $\mathbb{Z}/m$, for some $m\in\mathbb{N}$, or $\mathbf{F}_q$, $q$~some prime power.
Let $s$ be a unit in~$R$; then we define
$$
\varphi:\mathbb{B}_4\to F:=\varphi(\mathbb{B}_4)\subset GL(3,R)
$$
by considering the natural map $\mathbb{Z}\to R$ and specializing $t$ to~$s$.
Let us do it for $R=\mathbb{Z}/4$ and $s\equiv -1\bmod 4$.
We have:
\iftikziii
\[
\begin{tikzcd}[row sep=.0]
|[left]|\mathbb{B}_4\arrow[r,"\varphi"]&|[right]|F\subset\GL(3,\mathbb{Z}/4)\\
|[left]|\sigma_1\arrow[r, mapsto]&|[right]| 
\left(\begin{smallmatrix}
1 & 0 & 0 \\
1 & 1 & 0 \\
1 & 0 & 1
\end{smallmatrix}\right)\\
|[left]|\sigma_2\arrow[r, mapsto]&|[right]| 
\left(\begin{smallmatrix}
2 & 3 & 0 \\
1 & 0 & 0 \\
0 & 0 & 1
\end{smallmatrix}\right)\\
|[left]|\sigma_3\arrow[r, mapsto]&|[right]| 
\left(\begin{smallmatrix}
1 & 0 & 0 \\
0 & 2 & 3 \\
0 & 1 & 0
\end{smallmatrix}\right),
\end{tikzcd}
\]
\else
\begin{center}
\includegraphics{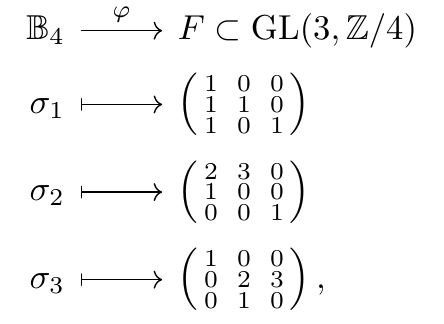}
\end{center}
\fi
where $F$ is a finite group of order~$768$. The classes of $\tilde{T}$ is
$$
[\tilde{T}]=
\left[\left(\begin{smallmatrix}
 2 & 2 & 3 \\
 0 & 3 & 1 \\
 3 & 0 & 3
 \end{smallmatrix}\right), \left(\begin{smallmatrix}
 3 & 3 & 0 \\
 0 & 1 & 3 \\
 2 & 1 & 0
 \end{smallmatrix}\right), \left(\begin{smallmatrix}
 3 & 3 & 2 \\
 3 & 0 & 2 \\
 1 & 0 & 1
 \end{smallmatrix}\right)\right]
$$
and the orbit of $[T]$ is
\begin{gather*}
\left[\left(\begin{smallmatrix}
 0 & 3 & 2 \\
 0 & 2 & 3 \\
 1 & 2 & 2
 \end{smallmatrix}\right), \left(\begin{smallmatrix}
 3 & 0 & 3 \\
 1 & 3 & 0 \\
 3 & 2 & 2
\end{smallmatrix}\right), \left(\begin{smallmatrix}
 1 & 0 & 1 \\
 0 & 2 & 1 \\
 0 & 1 & 1
\end{smallmatrix}\right)
 \right],
 \left[\left(\begin{smallmatrix}
  3 & 0 & 3 \\
  1 & 3 & 0 \\
  3 & 2 & 2
  \end{smallmatrix}\right), \left(\begin{smallmatrix}
  0 & 1 & 2 \\
  1 & 3 & 2 \\
  2 & 1 & 1
  \end{smallmatrix}\right), \left(\begin{smallmatrix}
  1 & 0 & 1 \\
  0 & 2 & 1 \\
  0 & 1 & 1
  \end{smallmatrix}\right)\right], \left[\left(\begin{smallmatrix}
  0 & 3 & 2 \\
  0 & 2 & 3 \\
  1 & 2 & 2
  \end{smallmatrix}\right), 
  \left(\begin{smallmatrix}
  1 & 0 & 1 \\
  0 & 2 & 1 \\
  0 & 1 & 1
  \end{smallmatrix}\right), \left(\begin{smallmatrix}
  2 & 3 & 0 \\
  1 & 3 & 2 \\
  0 & 3 & 3
  \end{smallmatrix}\right)\right],\\ \left[\left(\begin{smallmatrix}
  0 & 1 & 2 \\
  1 & 3 & 2 \\
  2 & 1 & 1
  \end{smallmatrix}\right), \left(\begin{smallmatrix}
  0 & 3 & 2 \\
  0 & 2 & 3 \\
  1 & 2 & 2
  \end{smallmatrix}\right), \left(\begin{smallmatrix}
  1 & 0 & 1 \\
  0 & 2 & 1 \\
  0 & 1 & 1
  \end{smallmatrix}\right)\right], \left[\left(\begin{smallmatrix}
  1 & 0 & 1 \\
  0 & 2 & 1 \\
  0 & 1 & 1
  \end{smallmatrix}\right), \left(\begin{smallmatrix}
  1 & 2 & 1 \\
  1 & 3 & 0 \\
  1 & 0 & 0
  \end{smallmatrix}\right), \left(\begin{smallmatrix}
  2 & 3 & 0 \\
  1 & 3 & 2 \\
  0 & 3 & 3
  \end{smallmatrix}\right)\right], \left[\left(\begin{smallmatrix}
  0 & 1 & 2 \\
  1 & 3 & 2 \\
  2 & 1 & 1
  \end{smallmatrix}\right), \left(\begin{smallmatrix}
  1 & 0 & 1 \\
  0 & 2 & 1 \\
  0 & 1 & 1
  \end{smallmatrix}\right), \left(\begin{smallmatrix}
  1 & 2 & 1 \\
  1 & 3 & 0 \\
  1 & 0 & 0
  \end{smallmatrix}\right)\right]\end{gather*}
It is easily checked that $\tilde{T}$ is not conjugate to any element of the orbit of~$T$.

The group $\mathbb{B}(A_3^r)$ is generated by $s_1^{-2}, s_2^{-2}, s_1\cdot s_2\cdot s_1^{-1}$;
they induce the following permutations in the orbit of~$T$:
$$
[(1,2,4)(3,5,6), (1,3,5)(2,4,6), (1,3,4)(2,5,6)],
$$
showing that it is actually an orbit. Since we have shown that the braid monodromies are
not conjugate, we deduce that no homeomorphism exists. The computations
have been done with \texttt{Sagemath}~\cite{sage66}
and \texttt{GAP4}~\cite{GAP4}.
\end{proof}


\subsection{Last comments}
\mbox{}

Eyral-Oka curves give rise to other arithmetic Zariski pairs, namely using the projections from the singular 
points of type $\mathbb{E}_6$ and $\mathbb{A}_2$. When projecting from a point of type $\mathbb{E}_6$ it does 
not matter which one because of the symmetry of the curves which exchanges both points. 
%
%
One can compute the braid monodromy factorizations using again the fact that they are strongly real curves. 
In these cases, it is more involved to prove that the braid monodromy factorizations are not equivalent. 
In a future work we will use the computed representations to distinguish the braid monodromies using diagonal 
representations.


\providecommand{\bysame}{\leavevmode\hbox to3em{\hrulefill}\thinspace}
\providecommand{\MR}{\relax\ifhmode\unskip\space\fi MR }
\providecommand{\MRhref}[2]{%
  \href{http://www.ams.org/mathscinet-getitem?mr=#1}{#2}
}
\providecommand{\href}[2]{#2}

\end{document}